\documentclass{amsart}
\usepackage[margin=1.5in]{geometry}

\usepackage{amsmath, amssymb, amsthm, amssymb}
\usepackage{amsfonts}
\usepackage{graphicx}
\usepackage[frame,ps,matrix,arrow,curve,rotate]{xy}
\usepackage{enumerate}
\usepackage{tikz}
\usepackage{hyperref}
\usepackage{siunitx}
\usepackage{enumitem}
\usepackage{xspace}
\usepackage{centernot}
\usepackage{blkarray}
\usepackage{tikz-cd}
\usepackage{float}

\newtheorem{theorem}{Theorem}[section]
\newtheorem{lemma}[theorem]{Lemma}

\newtheorem{corollary}[theorem]{Corollary}
\newtheorem{proposition}[theorem]{Proposition}

\theoremstyle{definition}
\newtheorem{definition}[theorem]{Definition}

\theoremstyle{remark}

\newcommand{\mc}[1]{\mathcal{#1}}

\newcommand{\res}{\!\!\upharpoonright}

\newcommand{\abs}[1]{\lvert#1\rvert}

\newcommand{\set}[1]{\{#1\}}

\renewcommand{\geq}{\geqslant}

\renewcommand{\leq}{\leqslant}
\renewcommand{\emptyset}{\varnothing}

\newcommand{\liff}{\leftrightarrow}
\newcommand{\limplies}{\rightarrow}

\newcommand{\es}{\emptyset}

\DeclareMathOperator{\id}{id}

\newcommand{\bN}{\mathbb{N}}
\newcommand{\N}{\mathbb{N}}

\newcommand{\WF}{\mathsf{WF}}
\newcommand{\CSB}{\mathsf{CSB}}
\newcommand{\PP}{\mathsf{PP}}
\newcommand{\WPP}{\mathsf{WPP}}
\newcommand{\ACLO}{\mathsf{AC}^{\mathsf{LO}}}
\newcommand{\ACWO}{\mathsf{AC}^{\mathsf{WO}}}
\newcommand{\ZF}{\mathsf{ZF}}
\newcommand{\ZFA}{\mathsf{ZFA}}
\newcommand{\ZFC}{\mathsf{ZFC}}

\newcommand{\DC}{\mathsf{DC}}
\newcommand{\CC}{\mathsf{CC}}
\newcommand{\AC}{\mathsf{AC}}
\newcommand{\NDS}{\mathsf{NDS}}

\begin{document}

\title{Cardinal Well-foundedness and Choice}
\author{Andreas Blass and Dhruv Kulshreshtha}

\maketitle

\begin{abstract}
We consider several notions of well-foundedness of cardinals in the absence of the Axiom of Choice. Some of these have been conflated by some authors, but we separate them carefully. We then consider implications among these, and also between these and other consequences of Choice. For instance, we show that the Partition Principle implies that all of our versions of well-foundedness are equivalent. We also show that one version, concerning surjections, implies the Dual Cantor--Schr\"oder--Bernstein theorem. It has been conjectured that well-foundedness, in one form or another, actually implies the Axiom of Choice, but this conjecture remains unresolved.
\end{abstract}

\section{Introduction}

Unless specified otherwise, we work in Zermelo--Fraenkel set theory ($\ZF$), and we use \cite{Jech73} as a general reference and \cite{HR98} as an additional resource. Because, in the absence of the Axiom of Choice, some sets may admit no well-ordering, we cannot use initial ordinals as cardinals. The manner in which cardinals are coded as sets, however,
will not matter for our purposes. For example, one may follow Jech in using Scott's trick \cite[Section 11.2]{Jech73}.

The Axiom of Choice ($\AC$) implies that cardinals are well-ordered. Conversely, if cardinals are well-ordered, or even linearly ordered, then $\AC$ holds. On the other hand, in the absence of $\AC$, the (standard injective) ordering of cardinals is a partial ordering. Even in this case, we can ask questions about the well-foundedness of cardinals. For example, it is still an open problem whether some version of well-foundedness of cardinals implies $\AC$.

In $\ZF$, there is an ambiguity in the notion of well-foundedness of cardinals. For example, consider the following statements:

\begin{enumerate}
\item[(1)] Given a non-empty set $X$ of cardinals, there is a $\kappa \in X$ such that for any $\lambda \in X$, if $\lambda \leq \kappa$ then $\kappa = \lambda$.
    \item[(2)] If $(\kappa_i)$ is an infinite sequence of cardinals such that $\kappa_0\geq \kappa_1 \geq \kappa_2\geq\cdots$, then there is some $n\in\bN$ such that $\kappa_n = \kappa_{n+1}$.
    \item[(3)] If $(A_i)$ is an infinite sequence of sets such that $\abs{A_0} \geq \abs{A_1} \geq \abs{A_2} \geq \cdots$, then there is some $n\in\bN$ such that $\abs{A_n} = \abs{A_{n+1}}$.
    \item[(4)] Given $(A_i)$, an infinite sequence of sets, and $(f_i)$, a corresponding sequence of injections $f_i : A_{i+1} \to A_i$, there is some $n \in \bN$ such that $\abs{A_n} = \abs{A_{n+1}}$.
    \end{enumerate}

Although $(2)$ or $(3)$ are commonly used to refer to well-foundedness (as described in Section \ref{sec: lit review}), any one of these could be regarded as expressing well-foundedness of cardinals. Note, however, that $(1)$ is different from the rest in the absence of Dependent Choice, and even $(2), (3)$, and $(4)$ split in the absence of Countable Choice\footnote{(1) has no analogous split because it doesn't mention sequences.}. Some authors appear to have conflated versions $(2)$ and $(3)$, which will be discussed in more detail in Section \ref{sec: lit review}. 

The above statements will be given a more formal label in Definition \ref{def: well-foundedness}. This is also where we will consider the relations between these and analogous statements involving the surjective ordering of cardinals. Indeed, if there is an injection from $X$ into $Y$, then there exists a surjection the other way, unless $X=\es\neq Y$. Even this trivial construction becomes interesting when we talk about sequences, as in $(4)$ above: Given a sequence of injections $(f_i : X_i \to Y_i)$, one can certainly rely on Countable Choice to get a sequence of surjections $(g_i : Y_i \to X_i)$; and Countable Choice turns out to be necessary to obtain this conclusion. We show in Theorem \ref{thm: CC equivalence} that Countable Choice is equivalent to this triviality and even to an apparently weaker form of it.

We also show in Theorem \ref{thm: surj WF implies Dual CSB} that the analogous statement to (4), with a sequence of sets $(A_i)$ and a corresponding sequence of surjections $(f_i : A_i \to A_{i+1})$, implies the Dual Cantor--Schr\"oder--Bernstein theorem, i.e.\ the statement ``if $X$ and $Y$ are non-empty sets and there are surjections $g: X\to Y$ and $h: Y\to X$, then there is a bijection between $X$ and $Y$." Corollary \ref{cor: surj WF consequences} highlights some consequences of this result.

Finally, we consider connections between these forms of well-founded\-ness and other consequences of $\AC$. For example, the statement ``if $X = \es$ or $Y$ surjects onto $X$, then $X$ injects into $Y$," called the Partition Principle, implies that all the forms of well-foundedness (Definition \ref{def: well-foundedness}) are equivalent. The question of whether the Partition Principle implies $\AC$ is still open, and an answer to that may provide some insight on tackling the analogous question for well-foundedness. We also look at connections with Dependent Choice and the Dual Cantor--Schr\"oder--Bernstein theorem, among others.

% {\red TO DO:} The non-splitting of WF1 refers to the following: We originally had WF2 but then realized that it might make a difference whether we had (1) a decreasing sequence of cardinals or (2) a sequence of sets of those cardinalities or (3) a sequence of sets and functions witnessing the cardinality inequalities. So the original WF2 split into WF2, WF3, and WF4, which might be different in the absence of countable choice. WF1, on the other hand, is immune to this problem simply because it doesn't talk about an infinite sequence.

\section{Preliminaries}
\begin{definition} Given sets $X$ and $Y$, we say:
\begin{itemize}
    \item $\abs{X} = \abs{Y}$ if $X$ bijects onto $Y$.
    \item $\abs{X} \leq \abs{Y}$ if $X$ injects into $Y$.
    \item $\abs{X} \leq^* \abs{Y}$ if $X=\emptyset$ or $Y$ surjects onto $X$.
    \item $\abs{X} =^* \abs{Y}$ if $\abs{X}\leq^*\abs{Y}$ and $\abs{Y}\leq^*\abs{X}$.
\end{itemize}
\end{definition}
We use $<,\geq$, and $>$ in the customary way.

\begin{definition}[Abbreviations] We use the following abbreviations.
\begin{itemize}
    \item[$\AC$] The Axiom of Choice \cite[Form 1]{HR98}.
    \item[$\ACWO$] Choice for well-ordered families of non-empty sets \cite[Form 40]{HR98}.
    \item[$\CC$] Choice from countably infinite families \cite[Form 8]{HR98}.
    \item[$\CSB$] Cantor--Schr\"oder--Bernstein: If $\abs{X}\leq \abs{Y}$ and $\abs{Y}\leq \abs{X}$ then $\abs{X} = \abs{Y}$. 
    \item[$\CSB^*$] Dual Cantor--Schr\"oder--Bernstein: If $\abs{X} =^* \abs{Y}$ then $\abs{X} = \abs{Y}$ \cite[Form 168]{HR98}.
    \item[$\DC$] Dependent Choice: If $A$ is a nonempty set and $R$ is a relation on $A$ satisfying $(\forall x\in A)(\exists y\in A)(x\ R\ y)$, then there is a sequence $a_0, a_1, a_2, \ldots$ of elements of $A$ such that $(\forall n \in \bN)(a_n\ R\ a_{n+1})$ \cite[Form 43]{HR98}.
    \item[$\PP$] Partition Principle: 
    $\abs{X} \leq^* \abs{Y} \implies \abs{X} \leq \abs{Y}$ \cite[Form 101 A]{HR98}.
    \item[$\mathsf{WPP}$] Weak Partition Principle: If $\abs{X} \leq^* \abs{Y}$, then $\abs{Y} \not< \abs{X}$ \cite[Form 100]{HR98}.
\end{itemize}
\end{definition}
Of the above statements, $\CSB$ is a theorem of $\ZF$ (see \cite[Section 2.5]{Jech73}), whereas the remaining are consequences of $\AC$ but not provable in $\ZF$.

\begin{proposition}\label{prop: CSB etc.} Given sets $X$ and $Y$, the following are provable in $\ZF$:
\begin{enumerate}
\item[(1)] $\abs{X}\leq \abs{Y}$ and $\abs{Y}\leq \abs{X} \implies \abs{X} = \abs{Y}$,
    \item[(2)] $\abs{X} \leq \abs{Y} \implies \abs{X} \leq^* \abs{Y}$,
    \item[(3)] $\abs{X} \leq \abs{Y}$ and $\abs{Y}\leq \abs{X} \implies \abs{X} =^* \abs{Y}$,
    \item[(4)] $\abs{X}=\abs{Y}  \implies \abs{X} =^* \abs{Y}$.
\end{enumerate}
\end{proposition}

\begin{proof} First note that (1) is the statement of $\CSB$, which is provable in $\ZF$.\medskip

For (2), suppose that $g:X\to Y$ is an injection witnessing $\abs{X}\leq \abs{Y}$. If $X = \es$, then $\abs{X}\leq^*\abs{Y}$, by definition. On the other hand, if $X\neq \es$ we construct a surjection $h : Y \to X$ as follows. Since $g :X\to Y$ is injective, $g$ has an inverse function mapping $g[X]$ onto $X$. Extend that inverse by a constant map on the rest of $Y$ to get the desired $h$. (The fact that $X\neq\es$ is used to ensure the availability of a constant map.)
\medskip

(3) follows immediately from (2)    .
\medskip

For (4), suppose $\abs{X}=\abs{Y}$, and fix a bijection $f:X\to Y$. Then $f:X\to Y$ and $f^{-1}:Y\to X$ are surjections, which means $\abs{X} =^* \abs{Y}$.
\end{proof}

\section{Dedekind Sets}
\begin{definition}
    A set $X$ is said to be \textit{finite} if for some $n \in \bN = \set{0,1,2,\ldots}$, $\abs{X} = n$. $X$ is said to be \textit{infinite} otherwise.
\end{definition}

\begin{definition}
    A set $X$ is said to be \textit{Dedekind-finite} if there is no injection $\bN \to X$, and \textit{Dedekind-infinite} otherwise.
\end{definition}
The Dedekind-finiteness of $X$ can equivalently be characterized as ``for any $Y\subsetneq X$, $\abs{Y} < \abs{X}$" or ``every injection $X\to X$ is a bijection" (see \cite[Section 2.5]{Jech73}).

It is a theorem of $\ZF$ that every finite set is Dedekind-finite, and a theorem of $\ZFC$ (but strictly weaker than $\AC$) that every Dedekind-finite set is finite (see \cite[Form 9]{HR98}). On the other hand, it is consistent with $\ZF$ for there to exist an infinite Dedekind-finite set -- see the basic Cohen model in \cite[Chapter 5]{Jech73}. Such a set is said to be a \textit{Dedekind set}. 

% It is a theorem of $\ZF$ that every finite set is Dedekind-finite.\footnote{Note that it is a theorem of $\ZFC$, but strictly weaker than $\AC$, that every Dedekind-finite set is finite (see \cite[Form 9]{HR98}).} However, the converse is not necessarily true. It is a result of Cohen that if $\ZF$ is consistent, then it is consistent with $\ZF$ for there to exist an infinite Dedekind-finite set. Such a set is said to be a \textit{Dedekind set}.

\begin{theorem}[\cite{Tarski1965}]\label{thm: tarski no bijection} Suppose that there exists a Dedekind set. Then there are sets $X$ and $Y$ such that $\abs{X} \leq \abs{Y}$, $\abs{Y} \leq^* \abs{X}$, and $\abs{X}\neq\abs{Y}$.
\end{theorem}
\begin{proof} Fix a Dedekind set $S$. Let $Y$ be the set of all finite one-to-one sequences of elements in $S$ and $X$ be the set of all non-empty such sequences. $\abs{X} \leq \abs{Y}$ is witnessed by the inclusion map of $X$ to $Y$, and $\abs{Y} \leq^* \abs{X}$ is witnessed by the ``delete first element'' surjection from $X$ onto $Y$. So it remains only to prove that $|X|\neq|Y|$.

Suppose for a contradiction that $\abs{X}=\abs{Y}$. Since $X\subsetneq Y$, it follows that $Y$ is Dedekind-infinite. So, there is a sequence $(y_k)_{k\in\bN}$ of distinct elements of $Y$, i.e.\ distinct finite sequences from $S$.

% It suffices to show that $Y$ is Dedekind-finite: In this case if there were a bijection $Y \to X$, it would be an injection mapping $Y$ to a proper subset $X\subsetneq Y$, contradicting the Dedekind-finiteness of $Y$.

% So, seeking a contradiction, assume that $Y$ is Dedekind-infinite, witnessed by the injection $f:\bN\to Y$. Setting $y_n = f(n)$ gives a sequence $(y_n)_{n\in\bN}$ of distinct finite sequences from $S$.

We now construct a one-to-one sequence $\sigma$, which enumerates some of the elements of $S$. We then show that $\sigma$ must be infinite, contradicting the Dedekind-finiteness of $S$. Our construction of $\sigma$ will be inductive, gradually producing longer and longer finite initial segments.

Denote by $\sigma_k$ the finite initial segment of $\sigma$ produced after the $k$-th step of this construction, so that $\sigma_0 \subseteq \sigma_1 \subseteq \cdots \subseteq \sigma$, as follows: Let $\sigma_0$ be $y_0$. Let $\hat{y}_{n+1}$ be defined by taking $y_{n+1}$ and deleting any of its elements that have previously appeared in $\sigma_n$. Then $\sigma_{n+1}$ is given by appending $\hat{y}_{n+1}$ to the end of $\sigma_n$. As a special case, if all the elements of $y_{n+1}$ have already appeared in $\sigma_n$, $\hat{y}_{n+1}$ is the empty sequence, which is appended to give $\sigma_{n+1} = \sigma_n$.

First note that the sequence $\sigma$ must itself be one-to-one. This is because if there were a repeated element in $\sigma$, then its second appearance could not have been in the same $\hat{y}_k$ as its first appearance, since each $y_k$ is one-to-one. So, it must have been in a later $\hat{y}_l$, but this is impossible, because its first appearance would cause it to be deleted from any later $\hat{y}_l$.

Finally, to show that $\sigma$ is infinite, we will show that the sequence grows non-trivially at infinitely many steps, i.e.\ for each $k\in\bN$, there is $l\in\bN$ such that $\sigma_l\setminus\sigma_k$ is nonempty. Suppose that at stage $k$, $\sigma_k$ has $n$ elements from $S$. There are only finitely many one-to-one sequences that use at most those $n$ elements. So, there is $l$ such that $\hat{y}_{k+l}$ is non-trivial, i.e.\ $\sigma_l\setminus\sigma_k$ is nonempty. So, $\sigma$ is infinite, contradicting the Dedekind-finiteness of $S$. This contradiction proves that $\abs{Y} \neq \abs{X}$.
\end{proof}

\begin{corollary}\label{cor: Dual CSB fails in ZF} $\CSB^*$ is not necessarily true in $\ZF$: There are sets $X$ and $Y$ such that $\abs{X} =^* \abs{Y}$ and $\abs{X}\neq \abs{Y}$.
\end{corollary}
\begin{proof} From Proposition \ref{prop: CSB etc.}, we know that $\abs{Y}\leq\abs{X}$ implies $\abs{Y}\leq^*\abs{X}$. The statement of Theorem \ref{thm: tarski no bijection} can then be translated into the above corollary.
\end{proof}

An alternative proof of Corollary \ref{cor: Dual CSB fails in ZF} proceeds via the known facts that $\CSB^*$ implies $\CC$ (in \cite{Hig95}, it has been shown that $\CSB^*$ implies $\ACWO$, which trivially implies $\CC$) and that $\CC$ is independent of $\ZF$ .

Corollary \ref{cor: Dual CSB fails in ZF} tells us that in $\ZF$, the existence of a bijection between two sets is strictly stronger than the existence of surjections in both directions. 

\section{Well-Foundedness} We define the following notions of well-foundedness that are all, a priori, distinct in $\ZF$. We use the notation $\WF_{yk}^x$, where the subscript $y \in \set{i,s,b}$ refers to the injective, surjective, or bijective comparison in the conclusion of the corresponding statement; and the superscript $x \in \set{i,s}$ indicates the same for the hypothesis. Moreover, the numerical subscripts $k \in \set{1,2,3,4}$ indicate the strength of the statement in terms of implication, i.e.\ the lower indexed statements imply the corresponding higher ones. Finally, the reason why we omit $\WF_{yk}^b$ will become obvious to the reader after reading the definitions below. 

\begin{definition}\label{def: well-foundedness} We formalize the notions of well-foundedness $\WF_{yk}^x$, by first considering the case of $\WF_{bk}^i$ as follows:
\begin{itemize}
    \item[$\WF_{b1}^i$] Given a non-empty set $X$ of cardinals, there is a $\kappa \in X$ such that for any $\lambda \in X$, if $\lambda \leq \kappa$ then $\kappa = \lambda$.
    \item[$\WF_{b2}^i$] If $(\kappa_i)$ is an infinite sequence of cardinals such that $\kappa_0\geq \kappa_1 \geq \kappa_2\geq\cdots$, then there is some $n\in\bN$ such that $\kappa_n = \kappa_{n+1}$.
    \item[$\WF_{b3}^i$] If $(A_i)$ is an infinite sequence of sets such that $\abs{A_0} \geq \abs{A_1} \geq \abs{A_2} \geq \cdots$, then there is some $n\in\bN$ such that $\abs{A_n} = \abs{A_{n+1}}$.
    \item[$\WF_{b4}^i$] Given $(A_i)$, an infinite sequence of sets, and $(f_i)$, a corresponding sequence of injections $f_i : A_{i+1} \to A_i$, there is some $n \in \bN$ such that $\abs{A_n} = \abs{A_{n+1}}$.
    \end{itemize}

    The remaining $\WF_{yk}^x$'s are defined similarly. For example,
    \begin{itemize}
        \item[$\WF_{s2}^s$] If $(\kappa_i)$ is an infinite sequence of cardinals such that $\kappa_0\geq^* \kappa_1 \geq^* \kappa_2\geq^* \cdots$, then there is some $n\in\bN$ such that $\kappa_n \leq^* \kappa_{n+1}$. 
        \item[$\WF_{i3}^i$] If $(A_i)$ is an infinite sequence of sets such that $\abs{A_0} \geq \abs{A_1} \geq \abs{A_2} \geq \cdots$, then there is some $n\in\bN$ such that $\abs{A_n} \leq \abs{A_{n+1}}$.
        \item[$\WF_{b4}^s$] Given $(A_i)$, an infinite sequence of sets, and $(f_i)$, a corresponding sequence of surjections $f_i : A_i \to A_{i+1}$, there is some $n \in \bN$ such that $\abs{A_n} = \abs{A_{n+1}}$.
    \end{itemize}
\end{definition}

% \begin{proposition}[$\WF_4$]
%     The following are equivalent:
%     \begin{enumerate}
%         \item If $(A_i)$ is a sequence of sets with $A_0 \supseteq A_1 \supseteq A_2 \supseteq \cdots$, then there is some $n \in \bN$ such that $\abs{A_n}\leq \abs{A_{n+1}}$.
%         \item If $(A_i)$ is an infinite sequence of sets, and $(f_i)$ is a sequence of injections $f_i : A_i \to A_{i+1}$, then there is some $n \in \bN$ such that $\abs{A_n}\leq \abs{A_{n+1}}$.
%     \end{enumerate}
% \end{proposition}
% \begin{proof}
%     {\color{red} TO DO:}
% \end{proof}

\begin{proposition}\label{prop: up to down implications} Fix $x \in \set{i,s}$ and $y \in \set{i,s,b}$. For $j<k$ we have \[\WF_{yj}^x \limplies \WF_{yk}^x.\]
\begin{proof} We treat the case of 
$\WF_{bk}^i$ first. The proofs for the remaining $\WF_{yk}^x$'s can be done similarly.

The $\WF_{b2}^i \limplies \WF_{b3}^i \limplies \WF_{b4}^i$ implications are straightforward; more specifically, each subsequent statement has a stronger hypothesis with the same conclusion. We prove the most non-trivial implication $\WF_{b1}^i \limplies \WF_{b2}^i$ using the contrapositive. Suppose $(\kappa_i)$ were an infinite sequence of cardinals that violated $\WF_{b2}^i$, i.e.\ $\kappa_0 > \kappa_1 > \kappa_2 > \cdots$, then the set $X = \set{\kappa_i}$ violates $\WF_{b1}^i$. 
\end{proof}

\end{proposition}

\begin{proposition}
    The following are equivalent:
    \begin{enumerate}
        \item[(1)] $\WF_{i2}^i$, i.e.\ if $(\kappa_i)$ is an infinite sequence of cardinals such that $\kappa_0\geq \kappa_1 \geq \kappa_2\geq\cdots$, then there is some $n\in\bN$ such that $\kappa_n \leq \kappa_{n+1}$.
        \item[(2)] If $(\kappa_i)$ is an infinite sequence of cardinals such that $\kappa_0\geq \kappa_1 \geq \kappa_2\geq\cdots$, then there is some $n\in\bN$ such that for every $m\geq n$, $\kappa_m\leq \kappa_{m+1}$.
    \end{enumerate}
\end{proposition}
\begin{proof} The $(2) \implies (1)$ direction is trivial. For the $(1) \implies (2)$ direction, suppose $(\kappa_i)$ is a sequence as in the hypothesis of $(2)$, and that for every $n \in \bN$, there is some $m \geq n$ such that $\kappa_m > \kappa_{m+1}$.\footnote{Note that to get this ``$\forall n, \exists m \geq n$" statement from the hypothesis in $(2)$, which is only ``$\exists n$," we apply $(2)$ to an appropriate tail of the given sequence.} Now construct the sequence $(\lambda_i)$ as follows: Taking $n=0$, there is some $m\geq n$ such that $\kappa_m > \kappa_{m+1}$. Set $\lambda_0 = \kappa_m$ and $\lambda_1 = \kappa_{m+1}$. Now taking $n'=m+1$, there is some $m'\geq n'$ such that $\kappa_{m'} > \kappa_{m'+1}$. Pick the least such $m'$ and set $\lambda_2 = \kappa_{m'+1}$, and so on. So, for each $n\in \bN$, $\lambda_n > \lambda_{n+1}$. Thus, the sequence $(\lambda_i)$ violates $(1)$, completing our proof.  
\end{proof}
The analogous results can be proven similarly for the remaining $\WF_{xk}^x$'s, for $k>1$. 

\begin{proposition}\label{prop: specified injection version equivalent to subset}
    The following are equivalent:
    \begin{enumerate}
        \item[(1)] $\WF_{b4}^i$, i.e.\ given $(A_i)$, an infinite sequence of sets, and $(f_i)$, a corresponding sequence of injections $f_i : A_{i+1} \to A_i$, there is some $n \in \bN$ such that $\abs{A_n} = \abs{A_{n+1}}$.
        \item[(2)] Given an infinite nested sequence of sets $B_0 \supseteq B_1 \supseteq B_2 \supseteq \cdots$, there is some $n \in \bN$ such that $\abs{B_n} = \abs{B_{n+1}}$.
    \end{enumerate}
\end{proposition}
\begin{proof} %For the $(1) \implies (2)$ direction, suppose $(B_i)$ is an infinite nested sequence of sets $B_0 \supseteq B_1 \supseteq B_2 \supseteq \cdots$, such that for all $n \in \bN$, $\abs{B_n} > \abs{B_{n+1}}$. Then taking $A_i = B_i$, and $f_i : A_{i+1} \to A_i$ to be the inclusion map $A_{i+1} \subseteq A_i$, we have that for each $n \in \bN$, $\abs{A_n} = \abs{B_n} > \abs{B_{n+1}} = \abs{A_{n+1}}$, violating $(1)$.

For the $(1) \implies (2)$ direction, suppose $(B_i)$ is as in $(2)$, with $B_0 \supseteq B_1 \supseteq B_2 \supseteq \cdots$. Taking $A_i = B_i$, and $f_i : A_{i+1} \to A_i$ to be the inclusion map, by (1), there is an $n \in \bN$ such that $\abs{B_n} = \abs{A_n} = \abs{A_{n+1}} = \abs{B_{n+1}}$.

For the $(2) \implies (1)$ direction, suppose $(f_i : A_{i+1} \to A_i)$ is a sequence of injections as in $(1)$. We define a sequence $g_i : A_i \to A_0$, by setting $g_0 = \id_{A_0}$ and $g_{n+1} = g_n \circ f_n$. Then, taking $B_n = g_n[A_n]$ gives us a sequence $B_0 \supseteq B_1 \supseteq B_2 \supseteq \cdots$. Since each $f_i$ is injective, for each $n\in\bN, \abs{B_n} = \abs{A_n}$. It follows from $(2)$ that there is $n \in \bN$ such that $\abs{B_n}=\abs{B_{n+1}}$; so, for this $n$, $\abs{A_n} = \abs{A_{n+1}}$, completing our proof. 
\end{proof}
The analogous results can be proven similarly for the remaining $\WF_{y4}^x$'s. Note that in the case of surjective inequalities, ``analogous" means that $B_{n+1}$ is the quotient of $B_n$ by some equivalence relation.
\section{Cardinal Representatives}
We use the following results of Pincus from \cite{Pin74} to indicate why $\WF^x_{y2}$ might, for all we know, be different from $\WF^x_{y3}$. More specifically, it follows from Theorem \ref{thm: no card rep} that in $\ZF$, $\WF^x_{y2}$ can't be deduced from $\WF^x_{y3}$ by just ``choosing a representative" for each cardinal.

Let $C$ be the class of cardinal numbers defined in the standard way. In \cite{Pin74}, Pincus defines a class of \textit{cardinal representatives} $R$ as a class of sets such that for each $\mu \in C$, there is a unique $x \in R$ satisfying $\abs{x} = \mu$. In this regard, he proves the following.

\begin{theorem}[Theorem 1.1 of \cite{Pin74}]\label{thm: no card rep} It is consistent with $\ZF$ that no class of cardinal representatives exists.
\end{theorem}

Pincus describes the class $\mathsf{DR}$ of cardinals $\mu$ such that $\mu = \abs{x}$ for some ordinal definable $x$ (see \cite{MS71}). $\mathsf{DR}$ then has a definable class of representatives: represent $\mu$ by the $x$ of least definition\footnote{i.e.\ the first $x$ in the standard well-ordering of the ordinal definable sets} satisfying $\abs{x} = \mu$. Furthermore, $\mathsf{Df}$ is defined as the class of cardinals $\mu$ satisfying $\aleph_0 \not\leq \mu$, i.e.\ $\mathsf{Df}$ is the class of Dedekind-finite cardinals. Using this, Pincus proves the following theorem, which not only implies Theorem \ref{thm: no card rep} -- and hence that $\WF^x_{y2}$ can't be deduced from $\WF^x_{y3}$ by just choosing representatives for each cardinal -- but also implies that representatives can't even be chosen for the ``obvious" decreasing sequence that starts with a Dedekind cardinal and subtracts natural numbers.

\begin{theorem}[Theorem 3 of \cite{Pin74}] It is consistent with $\ZF$ that there exists $\mu \in \mathsf{DR} \cap \mathsf{Df}$ such that $\set{\mu - n}_{n \in \bN}$ has no class of cardinal representatives.
\end{theorem}

\section{Another Countable Choice Principle}
We introduce an equivalent formulation of $\CC$, which we will use in subsequent sections (particularly Sections \ref{sec: filling the gaps} and \ref{sec: additional assumptions}).

\begin{theorem}\label{thm: CC equivalence} The following are equivalent:
\begin{enumerate}
    \item[(1)] $\CC$;
    \item[(2)] Given countable sequences of non-empty sets $(A_n)$ and $(B_n)$, and a countable sequence of injections $(f_n : A_n \to B_n)$, there exists a countable sequence of surjections $(g_n : B_n \to A_n)$;
    \item[(3)] Given countable sequences of non-empty sets $(A_n)$ and $(B_n)$, and a countable sequence of injections $(f_n : A_n \to B_n)$, there exists a countable sequence of maps $(h_n : B_n \to A_n)$.
\end{enumerate}
\end{theorem}
\begin{proof}
Assume $(1)$ and the hypothesis of $(2)$. Since each $A_n$ is non-empty, we may choose $a_n \in A_n$, for each $n\in\bN$. Then applying the same construction as in Proposition \ref{prop: CSB etc.} (2) to each injection $f_n$, we may define a sequence of surjections $g_n : B_n \to A_n$ by taking $g_n\res_{f_n[A_n]} = f_n^{-1}$ and for any $y \in B_n \setminus f_n[A_n]$, $g_n(y) = a_n$.

The $(2)\implies (3)$ implication is trivial, since each surjection $g_n : B_n \to A_n$, is itself a map $B_n \to A_n$, i.e.\ set $h_n = g_n$.

Finally, suppose $(3)$ is true, and that we are given a countably infinite family $\set{A_n:n\in\N}$ of non-empty sets. Fix some arbitrary set $z$, for example $z = \es$. Let $B_n = A_n \cup \set{z}$ and $f_n : A_n \to B_n$ be the inclusion map (which would just be the identity map if $z \in A_n$). By assumption, there is a sequence of maps $(h_n : B_n \to A_n)_{n\in\bN}$. So, $(h_n(z))_{n\in\bN}$ is our desired sequence with $h_n(z) \in A_n$ for all $n\in \N$, completing our proof. %{\blue Do we really need distinct elements $n$ to add to each $A_n$? Can’t we just take any object $g \notin \bigcup_n A_n$ and add that instead?} 
\end{proof}

\section{Filling the gaps}\label{sec: filling the gaps}
Although results such as Theorem \ref{thm: CC equivalence} and Theorem \ref{thm: no card rep} are helpful in distinguishing different $\WF$ notions in $\ZF$ (specifically Theorem \ref{thm: CC equivalence}~(2), which is not provable in $\ZF$, is needed for $\WF_{i4}^s\limplies \WF_{i4}^i$ and $\WF_{s4}^s\limplies \WF_{s4}^i$, and Theorem \ref{thm: no card rep} argues that it is not necessarily the case that $\WF^x_{y3} \limplies \WF^x_{y2}$) there are some implications that are immediately true in $\ZF$.

\begin{proposition}\label{prop: full figure} In $\ZF$, we have the following implications. 

\begin{figure}[H]
\centering
\[\begin{tikzcd}
	& {\WF_{b1}^i} && {\WF_{i1}^i} \\
	& {\WF_{b2}^i} && {\WF_{i2}^i} \\
	& {\WF_{b3}^i} && {\WF_{i3}^i} \\
	{\WF_{s1}^i} & {\WF_{b4}^i} && {\WF_{i4}^i} & {\WF_{i1}^s} & {\WF_{b1}^s} \\
	{\WF_{s2}^i} &&&& {\WF_{i2}^s} & {\WF_{b2}^s} \\
	{\WF_{s3}^i} &&&& {\WF_{i3}^s} & {\WF_{b3}^s} \\
	{\WF_{s4}^i} && {\WF_{s1}^s} && {\WF_{i4}^s} & {\WF_{b4}^s} \\
	&& {\WF_{s2}^s} \\
	&& {\WF_{s3}^s} \\
	&& {\WF_{s4}^s}
	\arrow[from=1-4, to=2-4]
	\arrow[from=2-4, to=3-4]
	\arrow[from=3-4, to=4-4]
	\arrow[from=7-3, to=8-3]
	\arrow[from=8-3, to=9-3]
	\arrow[from=9-3, to=10-3]
	\arrow[from=4-1, to=5-1]
	\arrow[from=5-1, to=6-1]
	\arrow[from=6-1, to=7-1]
	\arrow[from=7-3, to=4-1]
	\arrow[from=8-3, to=5-1]
	\arrow[from=9-3, to=6-1]
	\arrow[from=4-5, to=5-5]
	\arrow[from=6-5, to=7-5]
	\arrow[from=4-5, to=1-4]
	\arrow[from=5-5, to=2-4]
	\arrow[from=6-5, to=3-4]
	\arrow[from=4-5, to=7-3]
	\arrow[from=5-5, to=8-3]
	\arrow[from=6-5, to=9-3]
	\arrow[from=7-5, to=10-3]
	\arrow[from=4-6, to=4-5]
	\arrow[from=5-6, to=5-5]
	\arrow[from=6-6, to=6-5]
	\arrow[from=7-6, to=7-5]
	\arrow[from=4-6, to=5-6]
	\arrow[from=5-6, to=6-6]
	\arrow[from=6-6, to=7-6]
	\arrow[from=5-5, to=6-5]
	\arrow[from=1-2, to=4-1]
	\arrow[from=3-2, to=6-1]
	\arrow[from=4-2, to=7-1]
	\arrow[from=1-2, to=1-4]
	\arrow[from=2-2, to=2-4]
	\arrow[from=3-2, to=3-4]
	\arrow[from=4-2, to=4-4]
	\arrow[from=1-4, to=1-2]
	\arrow[from=2-4, to=2-2]
	\arrow[from=3-4, to=3-2]
	\arrow[from=4-4, to=4-2]
	\arrow[from=1-2, to=2-2]
	\arrow[from=2-2, to=3-2]
	\arrow[from=3-2, to=4-2]
	\arrow[from=2-2, to=5-1]
\end{tikzcd}\]
\caption{$\ZF$ implications}
\label{fig: full figure}
\end{figure}
\end{proposition}

\begin{proof} The up-to-down implications in Figure~\ref{fig: full figure} are covered in Proposition \ref{prop: up to down implications}. Each of the remaining implications in the figure is a consequence of one or more of the following results: [i] $\CSB$, [ii] The existence of an injection one way implies the existence of a surjection the other way (Proposition \ref{prop: CSB etc.}), [iii] Every bijection is itself an injection, and [iv] Every bijection is itself a surjection.

Choose any $k \in \set{1,2,3,4}, l\in \set{1,2,3}$. 
\begin{itemize}
    \item[$\bullet$] [iii] gives $\WF_{bk}^s\limplies \WF_{ik}^s$;
    \item[$\bullet$] [ii] gives $\WF_{ik}^s\limplies \WF_{sk}^s$, $\WF_{il}^s\limplies \WF_{il}^i$, and $\WF_{sl}^s\limplies \WF_{sl}^i$;
    \item[$\bullet$] [i] and [iii] give $\WF_{bk}^i\liff \WF_{ik}^i$;
    \item[$\bullet$] [iv] gives $\WF_{bk}^i\limplies \WF_{sk}^i$.
\end{itemize}  
\end{proof}

\section{Surjective Well-foundedness implies $\CSB^*$}
In this section we prove that $\WF^s_{b4}$ implies $\CSB^*$. We also indicate some consequences of this result.

\begin{theorem}\label{thm: surj WF implies Dual CSB}
    $\WF_{b4}^s \implies \CSB^*$.
\end{theorem}

\begin{proof} Assume $\WF_{b4}^s$, and suppose $X$ and $Y$ are (non-empty) sets with surjections $g: X \to Y$ and $h: Y \to X$. Define $(A_i)$ and $(f_i : A_i \to A_{i+1})$ as follows: For any $k \in \bN$, let $A_{2k} = X$ and $A_{2k+1} = Y$. Let $(f_i : A_i \to A_{i+1})$ be associated sequence of surjections given by $f_{2k} = g$ and $f_{2k+1} = h$. By $\WF_{b4}^s$, there is some $n \in \bN$ such that $\abs{A_n} = \abs{A_{n+1}}$, i.e.\ $\abs{X} = \abs{Y}$. 
\end{proof}

\begin{corollary}\label{cor: surj WF consequences} $\WF_{b4}^s$ implies each of $\mathsf{WPP}$, $\ACWO$, and $\DC$.
\end{corollary}
\begin{proof} Since $\WF_{b4}^s$ implies $\CSB^*$ (Theorem \ref{thm: surj WF implies Dual CSB}) and $\CSB^*$ clearly implies $\WPP$, it follows that $\WF_{b4}^s$ implies $\WPP$. It was shown by Higasikawa that $\WPP$ implies $\ACWO$ (see \cite{Hig95}), hence $\WF_{b4}^s$ implies $\ACWO$. Finally, since $\ACWO$ implies $\DC$  (\cite{Jensen67}; see \cite[Theorem 8.2]{Jech73} for the proof), we have $\WF_{b4}^s$ implies $\DC$.
\end{proof}

Further consequences of Theorem \ref{thm: surj WF implies Dual CSB} and Corollary \ref{cor: surj WF consequences} can be found in \cite[Part V.]{HR98} and \cite[Appendix 2]{Moo82}.

\section{Additional Assumptions}\label{sec: additional assumptions}
In this section we focus on how Figure~\ref{fig: full figure} collapses with the additional assumptions of $\CC$, $\DC$, $\AC$, $\CSB^*$, or $\PP$, over $\ZF$. Since we have already established the equivalence $\WF_{bk}^i \liff \WF_{ik}^i$, we omit the $\WF_{ik}^i$'s in this section.

\begin{proposition}[$\ZF + \CC$]\label{prop: ZF+CC1} For any $x \in \set{i,s}, y\in \set{i,s,b}$,
\[ \WF^x_{y2} \liff \WF^x_{y3} \liff \WF^x_{y4}. \]
\end{proposition}
\begin{proof} The left to right implications are provable in $\ZF$, as shown in Proposition \ref{prop: full figure}. For the right to left implications, we prove $\WF^i_{b4} \limplies \WF^i_{b3} \limplies \WF^i_{b2}$, and the remaining cases can be done similarly. More specifically, we show $\neg \WF^i_{b2} \limplies \neg \WF^i_{b3} \limplies \neg \WF^i_{b4}$.

Suppose $(\kappa_n)$ is a witness to the failure of $\WF^i_{b2}$, i.e.\ $\kappa_0 > \kappa_1 > \cdots$. Using $\CC$, we may choose a set $A_n$ of cardinality $\kappa_n$ for each $n$. So we have $\abs{A_0} > \abs{A_1} > \abs{A_2} >\cdots$, which means that $(A_n)$ is a witness to the failure of $\WF^i_{b3}$.

Now, suppose we are given $(A_n)$, a witness to the failure of $\WF^i_{b3}$, i.e.\ $\abs{A_0} > \abs{A_1} > \abs{A_2} >\cdots$. This means that for each $n$ there is an injection $A_{n+1} \to A_n$, i.e.\ there is a nonempty set of injections $I_n$ from $A_{n+1}$ to $A_n$. Using $\CC$, choose, for each $n$, an injection $f_n : A_{n+1} \to A_n$. The sequence $(A_n)$ along with the maps $(f_n : A_{n+1}\to A_n)$ violate $\WF^i_{b4}$.
\end{proof}

\begin{proposition}[$\ZF + \CC$]\label{prop: ZF+CC2} We have
\[\WF_{i4}^s\limplies \WF_{i4}^i \text{ and } \WF_{s4}^s\limplies \WF_{s4}^i,\]
\end{proposition}
\begin{proof}
    As stated before, $\WF_{i4}^s\limplies \WF_{i4}^i$ and $\WF_{s4}^s\limplies \WF_{s4}^i$, follow from Theorem \ref{thm: CC equivalence}(2), which is equivalent to $\CC$.
\end{proof}

\begin{corollary}[$\ZF + \CC$]\label{cor: ZF+CC} Figure~\ref{fig: full figure} collapses into the following.

\begin{figure}[H]
\centering
\[\begin{tikzcd}
	& {\WF_{b1}^i} \\
	& {\WF_{b(2,3,4)}^i} \\
	{\WF_{s1}^i} && {\WF_{i1}^s} & {\WF_{b1}^s} \\
	{\WF_{s(2,3,4)}^i} && {\WF_{i(2,3,4)}^s} & {\WF_{b(2,3,4)}^s} \\
	& {\WF_{s1}^s} \\
	& {\WF_{s(2,3,4)}^s}
	\arrow[from=5-2, to=6-2]
	\arrow[from=3-1, to=4-1]
	\arrow[from=5-2, to=3-1]
	\arrow[from=6-2, to=4-1]
	\arrow[from=3-3, to=4-3]
	\arrow[from=3-3, to=5-2]
	\arrow[from=4-3, to=6-2]
	\arrow[from=3-4, to=3-3]
	\arrow[from=4-4, to=4-3]
	\arrow[from=3-4, to=4-4]
	\arrow[from=1-2, to=3-1]
	\arrow[from=1-2, to=2-2]
	\arrow[from=2-2, to=4-1]
	\arrow[from=3-3, to=1-2]
	\arrow[from=4-3, to=2-2]
\end{tikzcd}\]
\caption{$\ZF+\CC$ implications}
\label{fig: ZF+CC}
\end{figure}
\end{corollary}
\begin{proof} This follows from Propositions \ref{prop: ZF+CC1} and \ref{prop: ZF+CC2}.
\end{proof}

\begin{lemma}[\cite{Spector80}; see Form 43 R of \cite{HR98}]\label{lem: Spector DC equiv} $\DC$ is equivalent to the statement ``any partial order $(A, <)$ with no infinite descending chains is well-founded, that is, for all $X$ such that $\emptyset \neq X \subseteq A$, $X$ has a minimal element." 
\end{lemma}

\begin{proposition}[$\ZF + \DC$]\label{prop: ZF+DC} For any $x \in \set{i,s}, y\in \set{i,s,b}$,
\[ \WF^x_{y1} \liff \WF^x_{y2} \liff \WF^x_{y3} \liff \WF^x_{y4}, \]
and so Figure~\ref{fig: ZF+CC} further collapses into the following.
% https://q.uiver.app/#q=WzAsNSxbMSwyLCIoXFxXRl97cygxLDIsMyw0KX1ecykiXSxbMCwxLCIoXFxXRl97cygxLDIsMyw0KX1eaSkiXSxbMiwxLCIoXFxXRl97aSgxLDIsMyw0KX1ecykiXSxbMywxLCIoXFxXRl97YigxLDIsMyw0KX1ecykiXSxbMSwwLCIoXFxXRl97YigxLDIsMyw0KX1eaSkiXSxbMCwxXSxbMiwwXSxbMywyXSxbNCwxXSxbMiw0XV0=
\begin{figure}[H]
\centering
\[\begin{tikzcd}
	& {\WF_{b(1,2,3,4)}^i} \\
	{\WF_{s(1,2,3,4)}^i} && {\WF_{i(1,2,3,4)}^s} & {\WF_{b(1,2,3,4)}^s} \\
	& {\WF_{s(1,2,3,4)}^s}
	\arrow[from=3-2, to=2-1]
	\arrow[from=2-3, to=3-2]
	\arrow[from=2-4, to=2-3]
	\arrow[from=1-2, to=2-1]
	\arrow[from=2-3, to=1-2]
\end{tikzcd}\]
\caption{$\ZF+\DC$ implications}
\label{fig: ZF+DC}
\end{figure}
\end{proposition}
\begin{proof} Since $\DC$ implies $\CC$, the only additional implication is $\WF^x_{y2}\limplies\WF^x_{y1}$, which follows easily from Lemma \ref{lem: Spector DC equiv}.
\end{proof}

\begin{proposition}[$\ZF + \AC$] All the notions of well-foundedness defined in Definition \ref{def: well-foundedness} are provable, and therefore equivalent. That is, for any $x \in \set{i,s}, y\in \set{i,s,b}$, and $k,l\in\set{1,2,3,4}$,
\[ \WF^x_{yk} \liff \WF^x_{yl}. \]
\end{proposition} 

\begin{proposition}[$\ZF + \CSB^*$] For any $k\in\set{1,2,3,4}$,
\[ \WF^s_{sk} \limplies \WF^s_{ik} \limplies \WF^s_{bk}, \]
and so Figure~\ref{fig: full figure} collapses into the following.
% https://q.uiver.app/#q=WzAsMjAsWzIsNCwiKFxcV0Zfe3MxfV5zKSJdLFsyLDUsIihcXFdGX3tzMn1ecykiXSxbMiw2LCIoXFxXRl97czN9XnMpIl0sWzIsNywiKFxcV0Zfe3M0fV5zKSJdLFswLDIsIihcXFdGX3tzMX1eaSkiXSxbMCwzLCIoXFxXRl97czJ9XmkpIl0sWzAsNCwiKFxcV0Zfe3MzfV5pKSJdLFswLDUsIihcXFdGX3tzNH1eaSkiXSxbNCwyLCIoXFxXRl97aTF9XnMpIl0sWzQsMywiKFxcV0Zfe2kyfV5zKSJdLFs0LDQsIihcXFdGX3tpM31ecykiXSxbNCw1LCIoXFxXRl97aTR9XnMpIl0sWzUsMiwiKFxcV0Zfe2IxfV5zKSJdLFs1LDMsIihcXFdGX3tiMn1ecykiXSxbNSw0LCIoXFxXRl97YjN9XnMpIl0sWzUsNSwiKFxcV0Zfe2I0fV5zKSJdLFsyLDAsIihcXFdGX3tiMX1eaSkiXSxbMiwxLCIoXFxXRl97YjJ9XmkpIl0sWzIsMiwiKFxcV0Zfe2IzfV5pKSJdLFsyLDMsIihcXFdGX3tiNH1eaSkiXSxbMCwxXSxbMSwyXSxbMiwzXSxbNCw1XSxbNSw2XSxbNiw3XSxbMCw0XSxbMSw1XSxbMiw2XSxbOCw5XSxbMTAsMTFdLFs4LDBdLFs5LDFdLFsxMCwyXSxbMTEsM10sWzEyLDhdLFsxMyw5XSxbMTQsMTBdLFsxNSwxMV0sWzEyLDEzXSxbMTMsMTRdLFsxNCwxNV0sWzksMTBdLFsxNiw0XSxbMTgsNl0sWzE5LDddLFsxNiwxN10sWzE3LDE4XSxbMTgsMTldLFsxNyw1XSxbMCw4XSxbMSw5XSxbMiwxMF0sWzMsMTFdLFsxMSwxNV0sWzEwLDE0XSxbOSwxM10sWzgsMTJdLFs4LDE2XSxbOSwxN10sWzEwLDE4XV0=
% https://q.uiver.app/#q=WzAsMTIsWzAsMCwiKFxcV0Zfe3MxfV5pKSJdLFswLDEsIihcXFdGX3tzMn1eaSkiXSxbMCwyLCIoXFxXRl97czN9XmkpIl0sWzAsMywiKFxcV0Zfe3M0fV5pKSJdLFs0LDAsIihcXFdGX3tiMX1ecykiXSxbNCwxLCIoXFxXRl97YjJ9XnMpIl0sWzQsMiwiKFxcV0Zfe2IzfV5zKSJdLFs0LDMsIihcXFdGX3tiNH1ecykiXSxbMiwwLCIoXFxXRl97YjF9XmkpIl0sWzIsMSwiKFxcV0Zfe2IyfV5pKSJdLFsyLDIsIihcXFdGX3tiM31eaSkiXSxbMiwzLCIoXFxXRl97YjR9XmkpIl0sWzAsMV0sWzEsMl0sWzIsM10sWzQsNV0sWzUsNl0sWzYsN10sWzgsMF0sWzEwLDJdLFsxMSwzXSxbOCw5XSxbOSwxMF0sWzEwLDExXSxbOSwxXSxbNCw4XSxbNSw5XSxbNiwxMF1d
\begin{figure}[H]
\centering
\[\begin{tikzcd}
	{\WF_{s1}^i} && {\WF_{b1}^i} && {\WF_{b1}^s} \\
	{\WF_{s2}^i} && {\WF_{b2}^i} && {\WF_{b2}^s} \\
	{\WF_{s3}^i} && {\WF_{b3}^i} && {\WF_{b3}^s} \\
	{\WF_{s4}^i} && {\WF_{b4}^i} && {\WF_{b4}^s}
	\arrow[from=1-1, to=2-1]
	\arrow[from=2-1, to=3-1]
	\arrow[from=3-1, to=4-1]
	\arrow[from=1-5, to=2-5]
	\arrow[from=2-5, to=3-5]
	\arrow[from=3-5, to=4-5]
	\arrow[from=1-3, to=1-1]
	\arrow[from=3-3, to=3-1]
	\arrow[from=4-3, to=4-1]
	\arrow[from=1-3, to=2-3]
	\arrow[from=2-3, to=3-3]
	\arrow[from=3-3, to=4-3]
	\arrow[from=2-3, to=2-1]
	\arrow[from=1-5, to=1-3]
	\arrow[from=2-5, to=2-3]
	\arrow[from=3-5, to=3-3]
\end{tikzcd}\]
\caption{$\ZF+\CSB^*$ implications}
\label{fig: ZF+CSB^*}
\end{figure}
\end{proposition} 
\begin{proof}
    We show this for the case of $k=3$, i.e.\ $\WF^s_{s3} \limplies \WF^s_{i3} \limplies \WF^s_{b3}$, and the remaining cases can be done similarly. Moreover, since the opposite implications are true in $\ZF$, it suffices to show that $\WF^s_{s3} \limplies \WF^s_{b3}$. Suppose $(A_i)$ is an infinite sequence of sets such that $\abs{A_0} \geq^* \abs{A_1} \geq^* \abs{A_2} \geq^* \cdots$. From $\WF^s_{s3}$, we have that for some $n\in\bN$, $\abs{A_n} \leq^* \abs{A_{n+1}}$, and then by $\CSB^*$, we have $\abs{A_n} = \abs{A_{n+1}}$, thus completing our proof. 
\end{proof}

\begin{proposition}[$\ZF + \PP$]\label{prop: ZF+PP}
For any $x,x' \in \set{i,s}, y,y'\in \set{i,s,b}$, and $k,k'\in\set{1,2,3,4}$,
\[ \WF^x_{yk} \liff \WF^{x'}_{y'k'} \]
and Figure~\ref{fig: ZF+DC} further collapses into the following, and thus to a single point.
% https://q.uiver.app/#q=WzAsNSxbMSwyLCIoXFxXRl97cygxLDIsMyw0KX1ecykiXSxbMCwxLCIoXFxXRl97cygxLDIsMyw0KX1eaSkiXSxbMiwxLCIoXFxXRl97aSgxLDIsMyw0KX1ecykiXSxbMywxLCIoXFxXRl97YigxLDIsMyw0KX1ecykiXSxbMSwwLCIoXFxXRl97YigxLDIsMyw0KX1eaSkiXSxbMCwxXSxbMiwwXSxbMywyXSxbNCwxXSxbMiw0XSxbMSw0XSxbNCwyXSxbMiwzXSxbMCwyXSxbMSwwXV0=
\begin{figure}[H]
\centering
\[\begin{tikzcd}
	& {\WF_{b(1,2,3,4)}^i} \\
	{\WF_{s(1,2,3,4)}^i} && {\WF_{i(1,2,3,4)}^s} & {\WF_{b(1,2,3,4)}^s} \\
	& {\WF_{s(1,2,3,4)}^s}
	\arrow[from=3-2, to=2-1]
	\arrow[from=2-3, to=3-2]
	\arrow[from=2-4, to=2-3]
	\arrow[from=1-2, to=2-1]
	\arrow[from=2-3, to=1-2]
	\arrow[from=2-1, to=1-2]
	\arrow[from=1-2, to=2-3]
	\arrow[from=2-3, to=2-4]
	\arrow[from=3-2, to=2-3]
	\arrow[from=2-1, to=3-2]
\end{tikzcd}\]
\caption{$\ZF+\PP$ implications}
\label{fig: ZF+PP}
\end{figure}
\end{proposition}

\begin{proof}
    From Theorem \ref{thm: PP equiv conjucts}, we have $\PP\limplies\DC$, and clearly $\PP\limplies\CSB^*$, so the only additional implications are $\WF^i_{sk} \limplies \WF^i_{bk} \limplies \WF^s_{ik}$ and $\WF^i_{sk} \limplies \WF^s_{sk}$; and for the former, it suffices to show that $\WF^i_{sk} \limplies \WF^s_{ik}$. Again, we prove this for the $k=3$ cases, i.e.\ $\WF^i_{s3} \limplies \WF^s_{s3}$ and $\WF^i_{s3} \limplies \WF^s_{i3}$, and the rest are similar:
    
    Suppose $(A_i)$ is an infinite sequence of sets such that $\abs{A_0}\geq^*\abs{A_1}\geq^*\abs{A_2}\geq^*\cdots$. From $\PP$, we have $\abs{A_0}\geq\abs{A_1}\geq\abs{A_2}\geq\cdots$. Then from $\WF^i_{s3}$, it follows that there exists $n \in \bN$ such that $\abs{A_n}\leq^*\abs{A_{n+1}}$, proving $\WF^s_{s3}$; and then from $\PP$, we have $\abs{A_n}\leq\abs{A_{n+1}}$, proving $\WF^s_{i3}$.
\end{proof}

\section{Literature Review}\label{sec: lit review}

Although the $\WF$'s are not provable in $\ZF$, each $\WF$ is itself a theorem of $\ZFC$. As far as we can tell, it is unknown whether any of these $\WF$'s imply $\AC$. The notions of well-foundedness that have been studied the most are $\WF_{i2}^i$ -- see, for example, \cite{Karagila2012}, \cite[(B)]{Pel78}, and \cite[Form 7]{HR98} -- and $\WF_{i3}^i$ -- see, for example, \cite[$\NDS$]{HT2015}. As mentioned in the introduction, there are some instances where these forms have been conflated -- see, for example, \cite[$\NDS$]{BM90}, which is stated as $\WF_{i2}^i$, and then immediately reformulated as $\WF_{i3}^i$.

As also mentioned earlier, these forms of well-foundedness are generally brought up alongside the question of whether they imply $\AC$ -- see \cite{Pel78} and \cite{BM90}. Although Banaschewski and Moore seem to believe that the answer to this is almost certainly negative (see \cite[Open problems]{BM90}), Howard and Tachtsis argue that such a non-implication cannot be easily expected or accepted (see \cite[Introduction]{HT2015}). Moreover, Karagila notes that although there is no positive result yet, Theorem \ref{thm: Karagila theorem} suggests that $\AC$ might be equivalent to $\WF_{i2}^i$ (see \cite[Introduction]{Karagila2012}).

In \cite[Theorem 1]{HT2015}, Howard and Tachtsis showed that if there is a Dedekind set, then $\WF_{i3}^i$ fails. Their proof, however, also demonstrates a failure of $\WF_{i4}^i$, via their construction of an infinite strictly-nested sequence of sets $(S_x)$, where $x<y$ implies $S_x \subsetneq S_y$, and consequently $\abs{S_x} \neq \abs{S_y}$, since the $S_x$'s are Dedekind sets. They also showed that the converse is not necessarily true, i.e.\ the absence of Dedekind sets does not necessarily imply $\WF_{i3}^i$.

\begin{theorem}[See proof of Theorem 1 in \cite{HT2015}] $\WF_{i4}^i$ implies that every Dedekind-finite set is finite, i.e.\ there does not exist a Dedekind set. 
\end{theorem}

Karagila has pointed out that, by using Tarski's construction in \cite{Tarski1965} of a Dedekind-finite family of sets of order-type of the reals, all of which are $=^*$ to each other, $\WF^s_{i4}$ also implies that there are no Dedekind sets.

Let $\ACLO$ denote choice for linearly ordered families of non-empty sets \cite[Form 202]{HR98}. Howard and Tachtsis use the variation $\mc{N}12(\aleph_1)$ of the Basic Fraenkel Model as defined in \cite{HR98} -- where the set of atoms is taken to have cardinality $\aleph_1$ and finite supports are replaced with countable ones -- to prove the following strong non-implication:

\begin{theorem}[Theorem 2 of \cite{HT2015}]\label{thm: ACLO and ACWO dont imply NDS}
$\ACLO$ does not imply $\WF_{i3}^i$ in $\ZFA$. Hence, $\ACWO$ does not imply $\WF_{i3}^i$ in $\ZFA$ either.
\end{theorem}

Using a refinement by Pincus (see \cite{Pincus72})  of Jech and Sochor's Second Embedding Theorem (\cite{JS63}; see \cite[Theorem 6.8]{Jech73}), Howard and Tachtsis transfer the last assertion in Theorem \ref{thm: ACLO and ACWO dont imply NDS} to $\ZF$, thereby obtaining:

\begin{theorem}[Theorem 3 of \cite{HT2015}]\label{thm: HT2015 thm3}
$\ACWO$ does not imply $\WF_{i3}^i$ in $\ZF$.
\end{theorem}

They also obtain the following $\ZFA$ result (no transfer to $\ZF$ is mentioned).

\begin{theorem}[Theorem 4 of \cite{HT2015}]\label{thm: HT2015 thm4} The statement ``for all infinite cardinals $m$, $m + m = m$ \cite[Form 3]{HR98}" does not imply $\WF_{i3}^i$ in $\ZFA$. 
\end{theorem}

The anonymous referee noted that in view of Proposition~\ref{prop: specified injection version equivalent to subset}, and the proofs of Theorems~\ref{thm: ACLO and ACWO dont imply NDS},\ref{thm: HT2015 thm3}, and \ref{thm: HT2015 thm4} in \cite{HT2015}, Theorems~\ref{thm: ACLO and ACWO dont imply NDS},\ref{thm: HT2015 thm3}, and \ref{thm: HT2015 thm4} can be strengthened by replacing $\WF^i_{i3}$ by $\WF^i_{i4}$. Furthermore, by \cite[Proof B of Theorem 2]{HT2015}, Theorem~\ref{thm: ACLO and ACWO dont imply NDS} can be further strengthened by replacing $\WF^i_{i3}$ by $\WF^i_{s4}$ and $\WF^s_{s4}$.\footnote{In particular, it can be shown that $\WF^i_{s4}$ fails in the permutation model $M_{\aleph_1}$ of \cite[Proof B of Theorem 2]{HT2015} by slightly modifying the proof of \cite[Claim 3]{HT2015}. Then, since $\ACLO$ is true in $M_{\aleph_1}$, it follows from Proposition~\ref{prop: ZF+CC2} that $\WF^s_{s4}$ is also false in $M_{\aleph_1}$.}

The referee also brought to our attention the recent article of Tachtsis \cite{Tachtsis-NDS-2024} that substantially contributes to the study of the strength of well-foundedness (particularly, $\WF^i_{i3}$, equivalently $\WF^i_{b3}$, denoted $\NDS$ in the article) in the absence of $\AC$.

Let $\mathsf{BPIT}$ denote the Boolean Prime Ideal Theorem, i.e.\ that every Boolean algebra has a prime ideal \cite[Form 14]{HR98}, which is a very natural choice principle (see \cite[Section 2.3]{Jech73} and the list of equivalents to Form 14 in \cite{HR98}) and is strictly weaker than $\AC$ (see \cite[Section 7.1]{Jech73}). Tachtsis uses the Howard--Rubin model ($\mc{N}38$ as defined in \cite{HR98}), as well as a transfer theorem of Pincus, to prove the following strong non-implication:

\begin{theorem}[Theorem 6.4 of \cite{Tachtsis-NDS-2024}] $\mathsf{BPIT}+\CC$ does not imply $\WF^i_{i3}$ in $\ZF$.
\end{theorem}

The referee noted that the proof of \cite[Theorem 6.4]{Tachtsis-NDS-2024} shows that $\mathsf{BPIT}+\CC$ does not imply $\WF^i_{i4}$ in $\ZF$. Therefore, by Proposition~\ref{prop: ZF+CC2}, one further gets that $\mathsf{BPIT}+\CC$ does not imply $\WF^s_{i4}$. Moreover, by a minor modification of the proof of \cite[Theorem 6.1]{Tachtsis-NDS-2024}, i.e.\ the failure of $\WF^i_{i3}$ in $\mc{N}38$, one obtains that $\WF^i_{s4}$ is false in $\mc{N}38$. Since $\CC$ holds in $\mc{N}38$, by Proposition~\ref{prop: ZF+CC2}, one further gets that $\WF^s_{s4}$ is false in $\mc{N}38$. That is:

\begin{theorem}[See proofs of Theorems 6.1 and 6.4 of \cite{Tachtsis-NDS-2024}]
    $\mathsf{BPIT}+\CC$ does not imply $\WF^i_{s4}$ and $\WF^s_{s4}$ (hence, by Proposition~\ref{prop: full figure}, any $\WF$ principle) in $\ZF$.
\end{theorem}

For a cardinal $\kappa$, the Principle of Dependent Choice for $\kappa$, denoted $\DC_\kappa$ states that for every non-empty set $X$, if $R$ is a binary relation such that for every ordinal $\alpha < \kappa$, and every $f:\alpha \to X$ there is some $y \in X$ such that $f\ R\ y$, then there is $f : \kappa \to X$ such that for every $\alpha < \kappa$, $f\res\alpha\ R\ f(\alpha)$. In \cite{Karagila2012}, Karagila argues the relative strength of $\WF_{i2}^i$ by proving independence from $\DC_\kappa$ (for any $\kappa$). More specifically, Karagila argues that when a choice principle is not provable by $\DC_\kappa$, for any $\kappa$, it hints that said principle may be equivalent to the axiom of choice, or that it is ``orthogonal” to $\DC_\kappa$-like principles.

\begin{theorem}[See proof of Theorem 5.1 in \cite{Karagila2012}]\label{thm: Karagila theorem} 
For every cardinal $\kappa$, it is consistent with $\ZF + \DC_\kappa$ that for every ordinal $\alpha$ there is a decreasing sequence of cardinals of order type $\alpha^*$, i.e.\ $\alpha$ with the reversed ordering. In particular, $\DC_\kappa$ does not imply $\WF_{i2}^i$ in $\ZF$. 
\end{theorem}

Karagila has pointed out that the proof of Theorem \ref{thm: Karagila theorem}, in fact, establishes that $\DC_\kappa$ does not imply $\WF^s_{s4}$ and $\WF^i_{s4}$. Consequently, $\DC_\kappa$ does not imply $\WF^x_{yk}$ for any $x,y,k$.

Let $\mathsf{IP}$ or ``intermediate power" be the statement ``if $\abs{X} \leq^* \abs{Y}$ and $\abs{Y} \not\leq^* \abs{X}$, then there is a $Z$ such that $\abs{X} \leq^* \abs{Z}$ and $\abs{Z} < \abs{Y}$." In \cite{Pel78}, Pelc presents a proof, due to Pincus, of the following equivalence involving $\PP$ and $\CSB^*$.

\begin{theorem}[Theorem 7 of \cite{Pel78}]\label{thm: PP equiv conjucts} Assume $\WF_{i2}^i$. Then $\PP$ is equivalent to the conjunction of $\CSB^*$, $\DC$, and $\mathsf{IP}$.
\end{theorem}

Note that Pincus does not require the additional assumption of $\WF_{i2}^i$ to obtain each of the three conjuncts from $\PP$; the absence of a decreasing sequence of cardinalities is, however, used to prove the other direction. 

Pelc defines the notion of a full tree $T_m$ as follows: its minimal element is the cardinal $m$ and for $x \in T_m$, the immediate successors of $x$ are such cardinals $y$ that $2^y = x$. $T_m$ is full iff it has branches of arbitrary finite length. Pelc uses this notion in the following theorem.

\begin{theorem}[Theorem 9 of \cite{Pel78}]
    Let $m$ be an arbitrary cardinal and $T_m$ a full tree. Then $\WF_{i2}^i + \CC$ implies that in $T_m$ there is a $\leq^*$-antichain of cardinality $\aleph_0$.
\end{theorem}

\section{Open Problems}
The first, and possibly most obvious, style of open problems involves $\WF$ and $\AC$. As mentioned earlier, the question of whether some form of well-foundedness implies $\AC$ has been raised in several papers, including \cite{Pel78}, \cite{BM90}, \cite{Karagila2012}, and \cite{HT2015}. Through Definition \ref{def: well-foundedness}, this question breaks into a long list of questions, with the extremes being:
\begin{enumerate}
    \item[1.] $\WF^s_{b1} \implies \AC$?
    \item[2.] $\WF^i_{s4} \implies \AC$?
    \item[3.] $\WF^s_{s4} \implies \AC$?
\end{enumerate}
In this case, a negative answer to (1) or an affirmative answer to (2) or (3) will give us the most information in relation to the other $\WF$ forms. On the other hand, since $\WF^s_{b1}$ is relatively the strongest form on our list, it would likely be more capable of providing the machinery required to prove $\AC$.

Figure~\ref{fig: full figure} suggests another natural genre of questions: Are there any other implications without additional assumptions among the various well-foundedness principles? Are any implications provably irreversible? For example, for $k \in \set{1,2,3,4}$:
\begin{enumerate}
    \item[4.] $\WF^i_{bk} \implies \WF^s_{ik}$, $\WF^s_{sk} \implies \WF^s_{ik}$, or $\WF^s_{ik} \implies \WF^s_{bk}$?
    \item[5.] $\WF^i_{sk} \centernot\implies \WF^i_{bk}$ or $\WF^s_{ik} \centernot\implies \WF^s_{sk}$?
\end{enumerate}
Note that if there were forms $\WF_a$ and $\WF_b$ of well-foundedness with $\WF_a \implies \WF_b \centernot\implies \WF_a$, then $\WF_b \centernot\implies \AC$.

The third list of questions involves implications between $\WF$ and other choice principles. Some of these questions have previously been raised in the literature; see, for example, \cite{BM90} and \cite{Pel78}. In the following list, a question of the form ``$\WF \implies X$" is really asking whether there is some well-foundedness principle that implies $X$.
\begin{enumerate}
    \item[6.] $\WF \implies \CSB^*$, or vice versa?
    \item[7.] $\WF \implies \PP$, or vice versa?
    \item[8.] $\WF\implies \WPP$, or vice versa?
\end{enumerate}
Theorem \ref{thm: surj WF implies Dual CSB} partially answers (6) and (8) in the case of $\WF^s_{bk}$, but other cases remain open. Using \cite{HR98} as a reference, one could come up with many more such questions.

Since $\WF^i_{i4}$ implies that every Dedekind-finite set is finite, i.e.\ the trichotomy statement $\aleph_0$-$\mathsf{TC}$ ``for every $X$, $\abs{X} < \aleph_0$, $\abs{X} > \aleph_0$, or $\abs{X} = \aleph_0$," one may ask
\begin{enumerate}
    \item[9.] Is there a $\WF$ that implies $\aleph_1$-$\mathsf{TC}$, i.e.\ the trichotomy statement ``for every $X$, $\abs{X} < \aleph_1$, $\abs{X} > \aleph_1$, or $\abs{X} = \aleph_1$"?
\end{enumerate}
The referee notes that in view of Theorem~\ref{thm: Karagila theorem}, the discussion that follows Theorem~\ref{thm: Karagila theorem}, and the fact that $\DC_\kappa$ implies $\kappa$-$\mathsf{TC}$ (see \cite[Theorem 8.1(b)]{Jech73}), one obtains that $\kappa$-$\mathsf{TC}$ does not imply $\WF^s_{s4}$ and $\WF^i_{s4}$ (hence, by Proposition~\ref{prop: full figure}, any $\WF$ principle) for any cardinal $\kappa$. In particular, the converse implication does not hold.

Finally, since $\DC$ causes Figure~\ref{fig: full figure} to collapse vertically and $\PP$ causes said figure to collapse completely (see Propositions \ref{prop: ZF+DC} and \ref{prop: ZF+PP}), one may ask
\begin{enumerate}
    \item[10.] Is there some principle that causes a horizontal collapse of Figure~\ref{fig: full figure}, but not a vertical one? 
\end{enumerate}
Such a principle would not only have to be weaker than $\PP$, but would also have to be incomparable with $\DC$.

\textit{Acknowledgements.} We thank Ronnie (Ruiyuan) Chen for his helpful comments and questions at the Michigan Logic Seminar, Asaf Karagila for his helpful comments on the previous draft of this article, and Paul Howard for earlier conversations on this topic. We also thank the anonymous referee for many helpful suggestions and corrections, especially the comments that substantially improved the quality of Section 10.

\bibliography{References}
\bibliographystyle{alpha}

%Furthermore, given any poset $(P,\preceq)$, there is a model $\mc{U}$ of $\ZF$ and sets $(A_p)_{p\in P}$ in $\mc{U}$ such that $p\preceq q \iff \abs{A_p}\leq\abs{A_q}$ (see \cite[Theorem 11.1]{Jech73}).

\end{document}